\documentclass[12pt]{amsart}
\usepackage[headings]{fullpage}
\usepackage{amscd,amsmath,amssymb,amsfonts}
\usepackage{mathtools}
\usepackage{xcolor}
\usepackage[all]{xy}
\usepackage{hyperref}
\usepackage{url}
\usepackage{stmaryrd}
\usepackage{color}
\usepackage{enumitem}

\usepackage{tikz}
\usetikzlibrary{matrix,arrows,decorations.pathmorphing}
\usepackage[OT2,T1]{fontenc}
\DeclareSymbolFont{cyrletters}{OT2}{wncyr}{m}{n}
\DeclareMathSymbol{\Sha}{\mathalpha}{cyrletters}{"58}
\newcommand{\qr}[2]{\left(\frac{#1}{#2}\right)}
\newcommand{\matzz}[4]{\left(
\begin{array}{cc} #1 & #2 \\ #3 & #4 \end{array} \right)}
\theoremstyle{plain}
\newtheorem{thm}{Theorem}
\newtheorem{lem}[thm]{Lemma}
\newtheorem{cor}[thm]{Corollary}
\newtheorem{prop}[thm]{Proposition}

\newtheorem{defn}[thm]{Definition}

\theoremstyle{definition}
\newtheorem{ex}[thm]{Example}

\newtheorem{rmk}[thm]{Remark}

\numberwithin{thm}{section}

\theoremstyle{plain}
\newtheorem{thmABC}{Theorem}

\newtheorem{propABC}[thmABC]{Proposition}

\newtheorem{defnABC}[thmABC]{Definition}

\renewcommand{\tilde}{\widetilde}

\newcommand{\ga}[2]{\begin{gather}\label{#1}#2 \end{gather}}

\newcommand{\surj}{\twoheadrightarrow}
\newcommand{\inj}{\hookrightarrow}

\newcommand{\Hom}{{\rm Hom}}
\newcommand{\im}{{\rm im}}
\newcommand{\Spec}{{\rm Spec \,}}

\newcommand{\Gal}{{\rm Gal}}

\newcommand{\C}{{\mathbb C}}

\newcommand{\F}{{\mathbb F}}

\renewcommand{\P}{{\mathbb P}}
\newcommand{\Q}{{\mathbb Q}}

\newcommand{\I}{{\mathbb I}}
\newcommand{\rM}{{\mathrm{M}}}
\newcommand{\rH}{{\mathrm{H}}}
\newcommand{\rV}{{\mathrm{V}}}
\newcommand{\rT}{{\mathrm{T}}}

\newcommand{\Z}{{\mathbb Z}}

\newcommand{\res}{{\text{\sf res}\hspace{.1ex} }}

\DeclareMathOperator{\GL}{GL}
\DeclareMathOperator{\PGL}{PGL}

\DeclareMathOperator{\SL}{SL}

\DeclareMathOperator{\Aut}{Aut}
\DeclareMathOperator{\Out}{Out}
\DeclareMathOperator{\End}{End}
\DeclareMathOperator{\tr}{tr}

\newcommand{\pr}{\mathrm{pr}}
\newcommand{\one}{\mathbf{1}}
\newcommand{\ph}{\varphi}
\newcommand{\ab}{{\mathrm{ab}}}

\newcommand{\RoqFp}{C_{\F_p}}
\newcommand{\Roq}{C_{\bar \F_p}}
\newcommand{\RoqFpp}{C_{\F_{p^2}}}

\numberwithin{equation}{section}

\definecolor{shadecolor}{RGB}{186,238,186}
\definecolor{softlimegreen}{RGB}{186,238,186}
\definecolor{limegreen}{RGB}{208,243,208}
\definecolor{questioncolor}{RGB}{135, 173, 241} 
\definecolor{warningcolor}{RGB}{240,120,134}
\definecolor{verypaleyellow}{RGB}{255,255,194}

\begin{document}

\title[ Obstruction]{An obstruction to lifting to characteristic $0$  }
\author{H\'el\`ene Esnault, Vasudevan Srinivas \and  Jakob Stix  }
\address{Freie Universit\"at Berlin, Arnimallee 3, 14195, Berlin,  Germany}
\email{esnault@math.fu-berlin.de}
\address{   TIFR, School of Mathematics \\
Homi Bhabha Road \\
400005 Mumbai, India}
\email{ srinivas@math.tifr.res.in}
\address{  Goethe-Universit\"at, Robert-Mayer-Str. 6--10, 60325 Frankfurt, Germany }
\email{ stix@math.uni-frankfurt.de}
\thanks{ The second author (VS) was supported during part of the preparation of the article by a J. C. Bose Fellowship of the Department of Science and Technology, India. He also acknowledges support of the Department of Atomic Energy, India under project number RTI4001. The third author (JS) acknowledges support by Deutsche  Forschungsgemeinschaft  (DFG) through the Collaborative Research Centre TRR 326 "Geometry and Arithmetic of Uniformized Structures", project number 444845124.
}

\begin{abstract}  
We introduce a new obstruction to lifting  smooth proper varieties  in characteristic $p>0$  to characteristic $0$. It is based on Grothendieck's specialization homomorphism  and  the resulting  discrete finiteness properties of \'etale fundamental groups.
\end{abstract}

\maketitle

\section{Introduction} 

\subsection{The first example  and recent developments}
Let $A$ be a complete local noetherian domain with algebraically closed residue field $k$ and field of fractions $A \subset K$. In \cite{Ser61}, Serre considers for a smooth 
proper variety $X$ over $k$ the question whether $X$ lifts to a smooth proper $X_A$ over $\Spec(A)$ for some $A$ as above. To construct the first examples of varieties in characteristic $p$ that do not lift to characteristic $0$, he assumes that 
$X$ admits a finite Galois \'etale cover $Y \to X$ by a complete intersection $Y \inj \P^n$ of $\dim(Y) \geq 3$ such that the action of the Galois group $G$ extends to a linear action on projective space. 
It is then proven in \cite[Lemma]{Ser61}, relying on Grothendieck's isomorphism 
\begin{equation}
\label{eq:GrothendieckIso}
\pi_1(X)\xrightarrow{\cong} \pi_1(X_A)
\end{equation}
between the \'etale fundamental groups  as defined in \cite{SGA1}  and denoted by $\pi_1$ in this note,
 that a lift $X_A$ implies a lift of the linear $G$-action to $\rho_A: G\to \PGL_{n+1}(A)$. If $k$ has characteristic $p>0$ and $G$ has a `large' $p$-Sylow subgroup, then  the deformation $\rho_A$ cannot exist and the variety $X$ does not lift. 
 
Serre's pioneer examples and methods have been largely amplified since then. For example, van Dobben de Bruyn proved  in \cite[Thm.~2]{vDdB21} that if $X$ lifts to characteristic $0$ and is endowed with a morphism $X\to C$ where $C$ is a smooth projective curve of genus $\ge 2$, then the morphism itself lifts to characteristic $0$ after an  inseparable base change over $C$. This enabled him to find examples of smooth projective varieties $X$ such that no alteration of $X$ lifts to characteristic $0$, see  \cite[Thm.~1]{vDdB21}.

\subsection{The new obstruction}
In this note we construct a new obstruction to the existence of a lift to characteristic $0$. Let $\overline{K}$ be an algebraic closure of $K$, the field of fractions of $A$ as above, and let $X_{\overline{K} }$ be the corresponding geometric generic fibre of the deformation $X_A$.
Recall that Grothendieck's isomorphism \eqref{eq:GrothendieckIso}  
is the key  point  to define Grothendieck's specialization  homomorphism 
\[
{\rm sp}\colon \pi_1(X_{\overline{K} })  \to \pi_1(X)
\]
which is surjective and an isomorphism on the pro-$p'$-completion, see \cite[Exp.~XIII 2.10, Cor.2.12]{SGA1}.  
On the other hand, 
if $\bar \eta: \Spec(\C)\to \Spec(K) \to \Spec(A)$  is a complex generic point and $X_{\bar \eta} = X_A \times_{\Spec(A),\bar \eta} \Spec(\C)$, by the Riemann Existence Theorem \cite[Exp.~XII Thm.~5.1]{SGA1} the \'etale fundamental group  $\pi_1(X_{\bar \eta})$ is the profinite completion of the topological fundamental group $\pi_1^{\rm top}(X_{\bar \eta}(\C))$ and the base change homomorphism $\pi_1(X_{\bar \eta})\to \pi_1(X_{\overline{K}})$ is an isomorphism.   As $X_{\bar \eta}(\C)$ is homotopy equivalent to a finite $CW$-complex (e.g. Morse theory), the discrete group $\Gamma=\pi_1^{\rm top}(X_{\bar \eta}(\C))$ is finitely presented as a discrete group. 
Thus those data yield a finitely presented group $\Gamma$ together with a group homomorphism 
\[
\Gamma\to \pi_1(X)
\]
which is  surjective on the profinite completion and an isomorphism on the pro-$p'$-completion. In addition, those properties propagate naturally for any finite  \'etale cover $X_U \to X$ associated to a finite index open subgroup $U \subseteq  \pi_1(X)$. 

\smallskip

This suggests the following definition.

\begin{defnABC}[see Definition~\ref{defn:pprime}] \label{defnABC:pprime}
A profinite group $\pi$ is said to be  \textbf{$p'$-discretely finitely generated} (resp.\ \textbf{$p'$-discretely finitely presented})   if there is a finitely generated (resp.\ presented)  discrete group $\Gamma$ together with a group homomorphism 
\[
\gamma: \Gamma \to \pi
\]
such that
\begin{enumerate}[label=(\roman*)]
\item the profinite completion 
$\hat \gamma: \hat \Gamma \to \pi$ is surjective;
\item
for any open subgroup $U\subset \pi$ 
with $\Gamma_U := \gamma^{-1}(U)$ 
the restriction 
$\gamma_U: \Gamma_U \to U$ induces a continuous group isomorphism on pro-$p'$-completions 
\[
\gamma_U^{(p')}: \Gamma_U^{(p')}\to U^{(p')}.
\]
\end{enumerate}
\end{defnABC}

Thus Grothendieck's theory of specialization for fundamental groups implies the following.

\begin{propABC}[see Proposition~\ref{prop:crit}] \label{prop:liftingobstruction}
Let $X$ be a smooth proper scheme defined over an algebraically closed field $k$ of characteristic $p$. If 
$\pi_1(X)$ is not $p'$-finitely presented, for example if  $\pi_1(X)$ is not even $p'$-finitely generated, then $X$ is not liftable to characteristic $0$.
\end{propABC}

This is the announced obstruction to lifting based on discrete finiteness properties of the \'etale fundamental group. This notion is very flexible, for example, there are many possible $A$ and many possible complex generic points of ${\rm Spec}(A)$, thus many possible $\Gamma$ and homomorphisms $\Gamma\to \pi_1(X)$, see Remark~\ref{rmk:nonhomeomorphic}.  However, we prove that our obstruction indeed prevents certain varieties from lifting.

\begin{thmABC}[Main result, see Theorem~\ref{thm:ex} and Corollary~\ref{cor:nonlift}] \label{thm:main}
Let $k$ be an algebraically closed field of characteristic $p > 0$. 
Then there are smooth projective varieties $X$ over $k$ 
such that $\pi_1(X)$ is not even $p'$-discretely finitely generated. Thus, in particular, $X$ does not lift to charactaristic $0$. 
\end{thmABC}

As the properties of Grothendieck's specialization homomorphism hold also for smooth quasi-projective varieties over $A$ with a good compactification with a relative normal crossings divisor  at infinity with values in the tame \'etale fundamental group, we can apply the notion in this case as well. 

Let us remark at this point that the main theorem of \cite[Thm.~1.1]{ESS21} asserts that, as a profinite group, $\pi_1(X)$ where $X$ is smooth projective, and more generally $\pi_1^t(X)$ when $X$ is smooth quasi-projective and admits a good compactification, is a finitely presented profinite group. Thus Theorem~\ref{thm:main} shows as well  that in general there is no finitely presented discrete  group which can explain  the main result of {\it loc.\ cit.}

\subsection{Outline}
We now describe our method to prove Theorem~\ref{thm:main}. 
Over $k=\bar \F_p$,  let $C$ be a smooth projective curve of genus $g \geq 2$ with $G=\Aut(C)$ its finite group of automorphisms. Let $P$ be a simply connected variety on which $G$ acts freely. We define $X=(C\times_kP)/G$ where $G$ acts diagonally. Then $G$ is a finite quotient of $\pi_1(X)$ and the associated Galois cover $C\times_kP$ has fundamental group equal to $\pi_1(C)$.  If  $\pi_1(X)$ was $p'$-discretely finitely generated by some $\Gamma \to \pi_1(X)$, then for any prime number $\ell\neq p$ 
the action $\rho_\ell$ of $G$ on $\ell$-adic cohomology $H^1(C, \Q_\ell)$ would be defined over $\Q$,
see Proposition~\ref{prop:discretelyfgimpliesSchur}.
 {\it We construct a curve $C$ for which this rationality propriety fails.}
 
The representation  $\rho_\ell$ is faithful, see 
Proposition~\ref{prop:AUTfaithfulOnH1}. It turns out that the rationaliity property fails if for all $\ell\neq p$, the representation $\rho_\ell$ is absolutely irreducible, see Section~\ref{sec:curvesWithSchur}. Indeed, the absolute irreducibility  implies that $C$ is supersingular, see  Proposition~\ref{prop:ss},  and  by Proposition~\ref{prop:Schur} that for $\ell\neq p$  the character of $\rho_\ell$  is $\Z$-valued, but the Schur index of $\rho_\ell$ is $2$. This prevents $\rho_\ell$ to be defined over $\Q$.  It remains then to construct such a curve. We show that the Roquette curve defined in Section~\ref{sec:Roquettecurve} has the required property. For this we have to make explicit the structure of its group of automorphism, see Appendix~\ref{app:A}.

\bigskip

{\it Acknowledgements}:  We thank Kivan\c{c} Ersoy who showed us examples of finitely presented groups the image of which in the profinite completion is not finitely presented. This in particular gives one reason  why  in our Definition~\ref{defn:pprime}  of a $p'$-discretely  finitely presented group, we can definitely not assume $\Gamma\to \pi$ to be injective. 

\section{profinite groups with \texorpdfstring{$p'$}{p'}-approximation}

\subsection{Finiteness properties} Let $p$ be a prime number.  For any group $H$, the pro-$p'$-completion of $H$ is defined as 
\ga{}{ H^{(p')}:=\varprojlim_{H\surj Q} Q,   \notag}
where $H \surj Q$ ranges through all finite quotients  with order $|Q|$ coprime to $p$.  
In case $H$ is already a profinite group, then we only consider continuous quotients $H \surj Q$, i.e., with open kernel. 
If $\alpha: H_1\to H_2$ is a group homomorphism (continuous if the $H_i$ are profinite), we denote the induced continuous homomorphism between the pro-$p'$-completions
by 
\ga{}{ \alpha^{(p')}: H_1^{(p')}\to H_2^{(p')} \notag}

\begin{rmk}
Let $\Gamma$ be a discrete group. Recall that any presentation of $\Gamma = \langle S \ | \ R\rangle$ with set of generators $S$ and set of relations $R$ gives rise to a \textit{presentation complex} $X_{S,R}$ with a single $0$-cell $\ast$, a $1$-cell for each $s \in S$ and a $2$-cell for each relation in $R$, see e.g. \cite[Cor. 1.28]{hatcherAT} for a description of the attaching maps. It follows from \textit{loc.\ cit.}\ that naturally
\[
\pi_1^{\rm top}(X_{S,R},\ast) = \Gamma.
\]
The proof shows in particular that the fundamental group of a CW-complex with finitely many $1$-cells (resp.\ finite $2$-skeleton) is finitely generated (resp.\ of finite presentation).
\end{rmk}

Recall the following well known proposition, see e.g. \cite[Prop. 4.2] {LyndonSchupp} 
and \cite[Cor. 2.7.1, Cor. 2.8]{MagnusKarrassSolitar} 
for the forward direction, or \cite[p.77]{Nie21},  \cite[p.162]{Sch27} for the claim on finite generation.

\begin{prop}[Reidemeister--Schreier] 
\label{prop:fgfptower}
Let $\Gamma$ be a discrete group and let $\Gamma_\circ \subseteq \Gamma$ be a subgroup of finite index.
\begin{enumerate}
\item 
$\Gamma$ is finitely generated if and only if $\Gamma_\circ$ is finitely generated.

\item 
$\Gamma$ is finitely presented if and only if $\Gamma_\circ$ is finitely presented.
\end{enumerate}
\end{prop}
\begin{proof}
If $\Gamma$ is finitely generated (resp.\ finitely presented), then there is a presentation complex $X$ for $\Gamma$ with finitely many $1$-cells (resp.\ finite $2$-skeleton)  as the number of cells multiplies by the degree of the cover. The finite index subgroup $\Gamma_\circ$ agrees with the fundamental group of a finite covering space $Y \to X$. The complex $Y$ then also has finitely many $1$-cells (resp.\ a finite $2$-skeleton).
Thus  $\Gamma_\circ$ is also finitely generated (resp.\ finitely presented).

\smallskip

For the converse direction we assume $\Gamma_\circ$ is finitely generated by $u_1, \ldots, u_n \in \Gamma_\circ$. Then $\Gamma$ is finitely generated by the generators of $\Gamma_\circ$ and representatives $x_t$ for each coset $t \in \Gamma/\Gamma_\circ$.
Let now $\Gamma_\circ = \langle u_1, \ldots, u_n \  | \  r_1, \ldots, r_m\rangle$ be moreover finitely presented. We may assume that $\Gamma_\circ$ is normal by first passing to  $\bigcap_{t\in \Gamma/\Gamma_\circ}  x_t\Gamma_\circ x_t^{-1}$, which is also of finite index and thus finitely presented by the first part of the proof. There are $a_{s,t} \in \Gamma_\circ$ for all $s,t \in \Gamma/\Gamma_\circ$ such that 
\begin{equation}
\label{eq:rel1}
x_sx_t = a_{s,t}x_{st},
\end{equation}
and for all $t \in \Gamma/\Gamma_\circ$ and all $1 \leq i \leq n$ there are $b_{i,t} \in \Gamma_\circ$ 
\begin{equation}
\label{eq:rel2}
x_t u_i x_t^{-1} =  b_{i,t}. 
\end{equation}
We write $a_{s,t}$ and $b_{i,t}$ as words in the $u_i$. In this sense then $\Gamma$ is finitely presented by 
\[
\Gamma = \langle u_1, \ldots, u_n, x_t \ ; \ t \in \Gamma/\Gamma_\circ \ | \ r_1, \ldots, r_m,  \eqref{eq:rel1}, \eqref{eq:rel2} \rangle.
\]
Indeed, if we denote the right hand side by $\tilde\Gamma$, then there is a surjective group homomorphism $\tilde \Gamma \surj \Gamma$ because all relations of  the presentation of $\tilde \Gamma$ hold in $\Gamma$. 
Let $\tilde \Gamma_\circ$ be the subgroup of $\tilde \Gamma$ generated by the $u_i$. Then the natural map
\[
\Gamma_\circ \surj \tilde{\Gamma_\circ} \inj \tilde{\Gamma} \to \Gamma
\]
is the identity onto $\Gamma_\circ \subseteq \Gamma$. We may thus identify $\Gamma_\circ$ with $\tilde \Gamma_\circ$. 
Moreover, by \eqref{eq:rel1} and \eqref{eq:rel2} any element can be put in a form $u x_t$ with $u \in U$ and $t \in \Gamma/\Gamma_\circ$. So the index of $\Gamma_\circ = \tilde{\Gamma_\circ}$ in $\tilde{\Gamma}$ is less or equal to the index $(\Gamma:\Gamma_\circ)$. Therefore $\tilde \Gamma \to \Gamma$ is an isomorphism. 

\end{proof}

The profinite version of Proposition~\ref{prop:fgfptower} holds as well.

\begin{prop}
\label{prop:profinite_fgfptower}
Let $\pi$ be a profinite group and let $U \subseteq \pi$ be an open subgroup. Then the following holds.
\begin{enumerate}
\item 
$\pi$ is topologically finitely generated if and only if $U$ is topologically finitely generated.
\item 
$\pi$ is topologically finitely presented if and only if $U$ is topologically finitely presented.
\end{enumerate}
\end{prop}
\begin{proof}
If $U$ is topologically finitely generated (resp.\ finitely presented), then the same holds for $\pi$ with an analogous proof as in Proposition~\ref{prop:fgfptower}. 
For the converse direction in (1) we refer 
\cite[Prop.~4.3.1]{wilsonprofinite}. 
The converse direction in (2) follows from the criterion in \cite[Thm.~0.3]{Lub01} thanks to Shapiro's Lemma.
\end{proof}

Recall the central definition of this note from the introduction.

\begin{defn} \label{defn:pprime}
A profinite group $\pi$ is said to be  \textbf{$p'$-discretely finitely generated} (resp.\ \textbf{$p'$-discretely finitely presented})   if there is a finitely generated (resp.\ presented)  discrete group $\Gamma$ together with a group homomorphism 
\[
\gamma: \Gamma \to \pi
\]
such that
\begin{enumerate}[label=(\roman*)]
\item the profinite completion 
$\hat \gamma: \hat \Gamma \to \pi$ is surjective;
\item
for any open subgroup $U\subset \pi$ 
with $\Gamma_U := \gamma^{-1}(U)$ 
the restriction 
$\gamma_U: \Gamma_U \to U$ induces a continuous group isomorphism on pro-$p'$-completions 
\[
\gamma_U^{(p')}: \Gamma_U^{(p')}\to U^{(p')}.
\]
\end{enumerate}
\end{defn}

\begin{rmk} \label{rmk:pprim}
A $p'$-discretely finitely generated (resp.\ finitely presented) profinite group $\pi$ has in particular by definition the property that $\pi$ is finitely generated (resp.\ $\pi^{(p')}$ is finitely presented as an object of the category of  pro-$p'$ groups). 

\end{rmk} 

\begin{rmk} \label{rmk:surj}
The condition (i) implies that for any $U$ as in (ii), the map $\hat \gamma_U: \hat \Gamma_U \to  U$ is surjective as well. 
Indeed, we must show that for all open normal subgroups $V \subseteq U$ the composition $\Gamma_U \to U \to U/V$ is surjective. Cofinally among these $V$ are open subgroups that are even normal in $\pi$. Now $\Gamma \surj \pi/V$ is surjective by assumption, and the preimage of $U/V$ is $\Gamma_U$. 
\end{rmk}

\subsection{Finiteness properties of fundamental groups} Of primary interest for us are the  (tame) fundamental groups of smooth  projective varieties (resp.\ smooth varieties with a good compactification).
 
\begin{prop} \label{prop:crit}
Let $X$ be a smooth proper scheme defined over an algebraically closed field $k$ of characteristic $p$. If 
$\pi_1(X)$ is not $p'$-finitely presented, for example if  $\pi_1(X)$ is not even $p'$-finitely generated, then $X$ is not liftable to characteristic $0$.
\end{prop}
\begin{proof}
This follows from the example below which treats the more general case of a smooth variety with normal crossing compactification. Note that if $X$ lifts to characteristic $0$, then there exists in particular a lift $X_V$ as described below.
\end{proof}

\begin{ex} \label{ex:ex}
    Let  $V$ be a complete discrete valuation ring of mixed characteristic $(0,p)$ with residue field $k$. Let $\Spec(\C )\to \Spec(V)$ be a 
    geometric generic point.  Let $X_V$  be a smooth proper scheme  over $V$   such that there is a compactification $X_V\hookrightarrow \bar X_V$ over $V$, where $\bar X_V\setminus X_V$ is a relative normal crossings divisor.  
Let $\Gamma:=\pi_1^{\rm top}(X(\C))$ be the topological fundamental group, which is  finitely presented. We compose the comparison isomorphism between the profinitely completed topological fundamental group and the \'etale fundamental group $\pi_1(X_\C)$
\[
\pi_1^{\rm top}(X(\C)) \to \widehat{\pi_1^{\rm top}(X(\C))} \xrightarrow{\sim} \pi_1(X_\C)
\]
with  Grothendieck's specialization homomorphism \cite[Exp.~XIII 2.10]{SGA1}
\[
  {\rm sp} \colon \pi_1(X_\C) \to \pi_1^t(X_{\bar k}),
\]
where $\pi_1^t(X_{\bar  k})$ the \'etale tame fundamental group of $X$ over $\bar k$, to obtain a homomorphism
\[
\gamma \colon \Gamma \to \pi_1^t(X_{\bar k}).
\]
The argument for curves given in \cite[Exp.~XIII Cor.~2.12]{SGA1} extends mutatis mutandis\footnote{The key input is the more general \cite[Exp.~XIII, Cor.~2.8]{SGA1}.} to $X_V$ and shows that ${\rm sp}$ is surjective and ${\rm sp}^{(p')}$ is an isomorphism. It follows that $\hat \gamma$ is surjective and $\gamma^{(p')}$ is an isomorphism.

We now show that the pro-$p'$-isomorphism property also holds for finite index subgroups.  For any open subgroup $U\subset \pi^t(X)$ define the corresponding $f: X_U\to X$ which is  connected finite \'etale and tame and has $\pi^t(X_U)=U$.  By \cite[Thm.~18.1.2]{EGAIV(4)} the cover lifts\footnote{That is the surjectivity part of $\hat{\gamma}$.} to a connected cover $f_V: X_{U, V} \to X_{V}$ thus to 
  $f_{\C}: X_{U,\C}\to X_\C$ in such a way that 
  \ga{}{\xymatrix@M+1ex{ \ar[d]_{\rm inj} \Gamma_U=\pi_1^{\rm top} (X_U(\C)) \ar[r]^(0.6){\rm sp} & \ar[d]^{\rm inj}\pi^t_1(X_{U,\bar  k})\\
 \Gamma = \pi_1^{\rm top}(X(\C)) \ar[r]^(0.6){\rm sp} &  \pi_1^t(X_{\bar k})
   }\notag}
is a fibre square.  We conclude 
that $\pi_1^t(X_{\bar k})$ is $p'$-discretely finitely generated.
 \end{ex}

Recall from  \cite{ESS21}  that, as a profinite group, $\pi_1^t(X)$  is finitely presented. 
It is natural to ask whether without the liftability assumption, $\pi_1^t(X_{\bar k})$ is always $p'$-discretely finitely  presented.
We shall prove in Section~\ref{sec:ex}  that it is even not necessarily $p'$-discretely finitely generated, producing thereby {\it a new liftability obstruction, notably for smooth projective varieties}.

\begin{rmk}
\label{rmk:nonhomeomorphic}
For a given profinite group $\pi$ that is $p'$-discretely finitely presented, the discrete group $\Gamma$ that realizes the discrete finite presentation by $\Gamma \to \pi$ is not uniquely determined by $\pi$. 
 Serre constructs in \cite{Ser64} an algebraic variety $X$ over a number field $k$ that upon different complex embeddings $\sigma, \tau \colon k \to \C$ yields non-homeomorphic complex manifolds $X^\sigma(\C)$, $X^\tau(\C)$. Their algebraic origin shows that the \'etale fundamental groups $\pi_1(X^\sigma_\C) \simeq \pi_1(X^\tau_\C)$ are isomorphic, but their topological fundamental groups are not. 
\end{rmk}

\section{Independence of  \texorpdfstring{$\ell$}{l} and rationality} 

\subsection{Rationality of representations} 
Let $G$ be a finite group. We recall how to decide if a complex linear representation of $G$ is defined over $\Q$, see e.g. \cite[Chap.~12]{Serre:reps}. The ring of complex valued characters $R_G$ has subrings
\[
R_{G}(\Q) \subseteq \bar{R}_G(\Q) \subseteq R_G
\]
where $R_{G}(\Q)$ is the ring of characters defined over $\Q$, and $\bar{R}_G(\Q)$ is the ring of $\Q$-valued characters. Wedderburn's Theorem decomposes the group ring $\Q[G]$ 
of $G$ according to the distinct irreducible representations $V_i$ of $G$ in $\Q$-vector spaces as
\begin{equation}
\label{eq:wedderburn}
\Q[G] = \prod_{i=1}^r \End_{D_i}(V_i)
\end{equation}
with simple factors  isomorphic to matrix rings $\rM_{d_i}(D_i)$ over skew fields $D_i = \End_{G}(V_i)$ with center $K_i$. Let $\chi_i : G \to \Q$ be the character of $V_i$ as a $G$-representation over $\Q$. These $\chi_i$ form a basis of $R_G(\Q)$. 

Next, using the reduced trace $\End_{D_i}(V_i) \to K_i$ composed with an embedding $\sigma : K_i \inj \C$ instead, we obtain a complex character $\psi_{i,\sigma} : G \to \C$. 
The $\psi_{i,\sigma}$ for all $i$ and all $\sigma$ form a basis of $R_G$, and the $\psi_i = \sum_{\sigma} \psi_{i,\sigma}$ form a basis of $\bar{R}_G(\Q)$ according to \cite[Prop.~35]{Serre:reps}.
Now $\dim_{K_i}(D_i) = m_i^2$ is the square of the index of $D_i$ as a skew field over $K_i$. The Schur index of the representation $V_i$ is this $m_i$. By \cite[Chap.~12]{Serre:reps} we have
$\chi_i = m_i \psi_i$ and so 
\[
\bar{R}_G(\Q)/R_G(\Q) = \bigoplus_{i=1}^r \Z/m_i\Z.
\]
This means that  a general  complex valued character $\chi = \sum_{i,\sigma} d_{i,\sigma} \psi_{i,\sigma}$  arises from a representation defined over $\Q$ if and only if the following two conditions are satisfied:
\begin{enumerate}
\item
the character must be Galois invariant: the values lie in $\Q$, i.e., the coefficients $d_{i,\sigma}$ are independent of $\sigma$; say $\chi = \sum_i d_i \psi_i$, and 
\item the coefficients $d_i$ must be divisible by the Schur index $m_i$. 
\end{enumerate}
 
\begin{rmk}
Since $G$ is a finite group, any representation in a $\Q$-vector space stabilizes a $\Z$-lattice (e.g. the lattice $\Lambda = \sum_{s \in G} s \Lambda_0$ generated by the $G$-translates of any lattice $\Lambda_0$)
and hence is even definable over $\Z$. So integrality is no further constraint for a representation of a  finite group $G$.
\end{rmk}

\subsection{Independence of \texorpdfstring{$\ell$}{l}}

Let $\pi$ be a profinite group, and let $\ph : \pi \surj G$ be a finite quotient with kernel $U_\ph = \ker(\ph)$. We denote by $U_\ph^\ab$ its abelianization. Then conjugation induces a commutative diagram
\begin{equation}
\label{eq:outerGaction}
\xymatrix@M+1ex{\pi \ar[r] \ar@{->>}[d]^{\ph} & \Aut(U_\ph) \ar@{->>}[d]  \ar[r] & \Aut(U_\ph^\ab) \\
G \ar[r] & \Out(U_\ph) \ar[ur] & }
\end{equation}
If $\pi$ is finitely generated, then $U_\ph$ is finitely generated by Proposition~\ref{prop:profinite_fgfptower}.
We deduce  that $U_\ph^\ab$ is a finitely generated $\hat \Z$-module. 
 The resulting $G$-representations with values in finite dimensional $\Q_\ell$-vector spaces are denoted by 
\begin{equation}
\label{eq:theelladicrep}
\rho_{\ph,\ell} : G \to \GL(U_\ph^\ab \otimes \Q_\ell),
\end{equation}
with character 
\[
\chi_{\ph,\ell} = \tr(\rho_{\ph,\ell}) : G \to \Q_\ell.
\]

\begin{defn}
\label{defn:independence}
A profinite group $\pi$ is said to satisfy  \textbf{independence of $\ell$} with respect to a prime number $p$ if 
\begin{enumerate}[label=(\roman*)]
\item as a profinite group $\pi$ is finitely generated, and
\item \label{defn:independenceii} for all continuous finite quotients $\ph : \pi \surj G$ the following holds: for all $\ell \not= p$ the characters $\chi_{\ph,\ell}$ have values in $\Z$ and are independent of $\ell$.
\end{enumerate}
\end{defn}

\begin{rmk}
For a profinite group as in Definition~\ref{defn:independence} we define the following variant of condition \ref{defn:independenceii} 
\begin{enumerate}
\item[(ii')] for each $\ell\neq p$ fix an embedding $\Q_\ell \subset \C$. Then the $\rho_{\ph, \ell}$, viewed by scalar extension as  representations of $G$ in $\C$ vector spaces,  are all isomorphic for  all $\ell \neq p$. 
\end{enumerate}
Then  \ref{defn:independenceii}  is equivalent to (ii'). As $G$ is finite and $\C$ is of characteristic $0$, the representations are semisimple and thus determined by their characters. Consequently \ref{defn:independenceii}  implies (ii'). 

Conversely, if (ii') holds, then all characters $\chi_{\ph,\ell} : G \to \Q_\ell$ agree after composition with the chosen embedding $\Q_\ell \subset \C$ with a complex valued character $\chi : G \to \C$. Let $F \subseteq \C$ be the subfield generated by the values of $\chi$. This is an abelian number field since all eigenvalues are roots of unity. Moreover, the field $F$ is contained in $\Q_\ell \subset \C$ for all $\ell \not= p$, i.e., $F$ has a split place above $\ell$. It follows that $F/\Q$ is completely split over all $\ell \not= p$, and thus $F = \Q$ by Cebotarev's Theorem.
Therefore all $\chi_{\ph,\ell}$ take values in rational algebraic integers, i.e. in $\Z$, and these values independent of $\ell \not= p$.

We formulated condition (ii) rather than (ii') because of it suggests a motivic flavour.
\end{rmk}

\begin{prop}
\label{prop:discretelyfgimpliesindependenceofell}
Let $p$ be a prime number. Let $\pi$ be a profinite group which is $p'$-discretely finitely generated via $\Gamma \to \pi$. Then $\pi$ satisfies independence of $\ell$ with respect to $p$.
\end{prop}
\begin{proof}
Let $\ph: \pi \surj G$ be a finite continuous quotient.  The composite map $f: \Gamma \to G$ defines similarly with $\Gamma_\ph = \ker(f)$ and $\Gamma_{\ph}^\ab$ a representation
\ga{int}{
\rho_\ph : G \to \GL( \Gamma_{\ph}^\ab \otimes \Q)
}
in a finite dimensional $\Q$-vector space $\Gamma_{\ph}^\ab \otimes \Q$. The assumption on $\Gamma \to \pi$ yields that $\Gamma_\ph \to U_\ph$ is an isomorphism on pro-$p'$ completion, hence in particular for all $\ell \not= p$ 
 the homomorphism $\Gamma_\phi\to U_\phi$ induces a $G$-equivariant isomorphism
\[
\Gamma_{\ph}^\ab \otimes_{\Z} \Q_\ell = U_\ph^\ab \otimes_{\hat \Z} \Q_\ell.
\]
Thus the  $\rho_{\ph,\ell}$, for the various $\ell \not= p$, are compatible and {\it even definable over} $\Q$. 
\end{proof}

\begin{prop}
\label{prop:iindependenceofellgeometrriccase}
Let $k$ be an algebraically closed field, and let $X/k$ be a smooth proper variety over $k$. Then $\pi_1(X)$ satisfies independence of $\ell$ with respect to the characteristic of $k$. In case $k$ has characteristic $0$, the prime $p$ is arbitrary.
\end{prop}
\begin{proof}
If $X$ lifts to characteristic $0$, then the combination of Proposition~\ref{prop:crit} and Proposition~\ref{prop:discretelyfgimpliesindependenceofell} shows the claim.

Let now $k$ be of positive characteristic $p$ and let $X$ be arbitrary. Let $\ph : \pi_1(X) \surj G$ be a finite continuous quotient, and let $Y \to X$ be the corresponding $G$-Galois \'etale cover. Then the $G$-representation $U_\ph^\ab \otimes \Q_\ell$ is dual to the natural $G$-representation on $\rH^1(Y,\Q_\ell)$. 
There is  a scheme $S$ of finite type over $\F_p$ such that 
$X$ and the graphs ${\rm graph}(g)\subset X\times_k X$ have smooth proper models $X_S, \  {\rm graph}(g)_S\subset X_S\times_SX_S$.  By proper base change for \'etale cohomology \cite[Exp. XII, Thm.~5.1]{SGA4},  we reduce to the case where $k=\bar \F_p$. 

Let $\alpha: Y \to A$ be the Albanese morphism of $Y$. Then $\pi_1^\ab(Y) \surj \pi_1(A)$ is surjective with finite kernel and thus $\alpha^\ast : \rH^1(A,\Q_\ell) \to \rH^1(Y,\Q_\ell)$ is a $G$-equivariant isomorphism (note that $G$ does act on $A$ by automorphisms that do not necessarily fix the origin). We may thus replace $Y$ by $A$ and therefore in particular assume that $Y$ is  a smooth projective variety. As any $g\in G$ acts via correspondences, the characteristic polynomial of each $g$ acting on $\rH^1(Y,\Q_\ell)$ lies in $\Z[T]$ and is independent of $\ell$, see \cite[Thm.2.1 (2)]{KM74} (here we need the projectivity  rather than the properness assumption). 
\end{proof}

\subsection{The obstruction imposed by the Schur index}
Let $\pi$ be a profinite group that satifies independence of $\ell$ with respect to the prime number $p$. This means for a finite quotient $\ph: \pi \surj G$ that the character
\[
\chi_\ph = \chi_{\ph,\ell}  = \tr(\rho_{\ph,\ell}) : G \to \Q_\ell.
\]
has values in $\Z$ and is independent of $\ell \not= p$. This character $\chi_\ph$ belongs to $\bar{R}_G(\Q)$, and the \textbf{Schur index obstruction} in the proper sense is its class
\[
[\chi_\ph] \in \bar{R}_G(\Q)/R_G(\Q).
\]
This is the obstruction for the representation associated to $\chi_\ph$ to be actually defined as a linear representation of $G$ in a $\Q$ vector space.

\begin{defn}
We say that a profinite group $\pi$ which satifies independence of $\ell$ with respect to the prime number $p$ is \textbf{(Schur) rational} if for all finite continuous quotients $\ph: \pi \surj G$ the Schur index obstruction class $[\chi_\ph]$ is trivial, i.e., there is an actual $G$ representation in a $\Q$-vector space $V_\ph$ that gives rise, for all $\ell \not= p$, to the $\ell$-adic representations 
\[
U_\ph^\ab \otimes \Q_\ell \simeq V_\ph \otimes_{\Q} \Q_\ell.
\]
\end{defn}

The following proposition was actually proved within the proof given of Proposition~\ref{prop:discretelyfgimpliesindependenceofell}. 

\begin{prop}
\label{prop:discretelyfgimpliesSchur}
Let $\pi$ be a profinite group which is $p'$-discretely finitely generated.  Then $\pi$ satisfies independence of $\ell$ and moreover is rational.
\end{prop}

Not being Schur rational is inherited for fundamental groups in the following geometric context. We may focus on positive characteristic, because fundamental groups of smooth proper varieties in characteristic $0$ satisfies independence of $\ell$ and are rational due to Proposition~\ref{prop:discretelyfgimpliesSchur}.

\begin{prop}
\label{prop:productSchurobstructed}
Let $k$ be an algebraically closed field, and let $X$ and $Y$ be  smooth proper varieties over $k$.
If $\pi_1(X)$ is not Schur rational, then $\pi_1(X \times_k Y)$ is not Schur rational either and in particular $X \times_k Y$ does not lift to characteristic $0$.
\end{prop}
\begin{proof}
Let $\ph: \pi_1(X) \surj G$ be a finite quotient such that the corresponding character $\chi_\ph$ has non-trivial class in $\bar{R}_G(\Q)/R_G(\Q)$. 
As $X$ (in fact even $X$ and $Y$) are proper, we have the K\"unneth formula, see \cite[Exp.X, Cor. 1.7]{SGA1},
\[
\pi_1(X\times_k Y)= \pi_1(X) \times \pi_1(Y).
\]
Composition with the first projection $\ph \circ \pr_1 : \pi_1(X \times_k Y) \surj G$ leads to the character 
\[
\chi_{\ph \circ \pr_1} = \chi_\ph + \dim_{\Q_\ell} \rH^1(Y,\Q_\ell) \cdot \one_G
\]
where $\one_G$ is the trivial character of $G$. Because $\one_G$ is defined over $\Q$, it follows that $\chi_{\ph \circ \pr_1}$ has the same class in $\bar{R}_G(\Q)/R_G(\Q)$ as $\chi_\ph$. This proves the claim.
\end{proof}

\section{Curves with many automorphisms } \label{sec:manyaut}

\subsection{Action on \texorpdfstring{$\rH^1$}{first cohomology}}
In this section, we consider a specific curve $C$ defined over a finite field with a very large group $G$ of automorphisms, and we single out a property of  the representation of $G$ on its first $\ell$-adic cohomology  $H^1(C, \Q_\ell)$    which prevents a variety $X$ constructed in the style of Serre to lift to characteristic $0$.

We start with the well known fact that this action is faithful.

\begin{prop}
\label{prop:AUTfaithfulOnH1}
Let $C$ be a smooth projective curve of genus $g \ge 2$ over an algebraically closed field $k$. Then, for all $\ell$ different from the characteristic of $k$, the representation
\[
\rho_\ell \colon \Aut(C) \inj \GL\big(\rH^1(C,\Q_\ell)\big)
\]
is faithful.
\end{prop}
\begin{proof}
Let $g \in G$ be nontrivial and in the kernel. Then the graph of $g$ in $C \times_k C$ and the diagonal  intersect in a  scheme of dimension $0$, the degree of which we can compute cohomologically by the Grothendieck--Lefschetz formula as 
\[
|\textrm{degree of the fixed point scheme of } g \textrm{ on } C| = \tr\big(g^\ast | \rH^*(C,\Q_\ell)\big) = 2 - 2g < 0.
\]
This is absurd.
\end{proof}

\subsection{The Roquette curve}
\label{sec:Roquettecurve}
In \cite[\S4]{Roq70} Roquette defines the smooth projective curve $\RoqFp$ over $\F_p$ which is the normal
compactification of the affine curve defined by 
\ga{}{\RoqFp: \  y^2=x^p-x, \notag}
which we call {\it Roquette curve} in this note. The map $(x,y)\mapsto x$ defines $\RoqFp$ as a double cover $\RoqFp \to \P^1$. It follows that for $p=2$ the curve $\RoqFp$ is rational, and thus we shall consider only the case $p > 2$ from now on. For $p\neq 2$ the hyperelliptic cover $\RoqFp \to \P^1$ considered above is tame, and the Riemann--Hurwitz formula immediately yields the genus $g = g(\RoqFp)$ as
\ga{}{2g = p-1.\notag}
In particular, the Roquette curve has genus $\ge 2$ if and only of $p\ge 5$. 

We set $C = \Roq := \RoqFp \otimes \bar \F_p$. By \cite[\S4]{Roq70}, the group of automorphisms $\Aut(C)$ over $\bar \F_p$  is of cardinality equal to 
\[
|\Aut(C)| = 2 \cdot |\PGL_2(\F_p)| =  2p(p^2-1).
\] 
For $p \ge 5$,  the size of $\Aut(C)$ exceeds the Hurwitz bound $84(g-1)$, which bounds from above the order of automorphism groups of curves of genus $g \ge 2$ in characteristic $0$. Actually Roquette proved in \cite{Roq70} that among curves of genus $g$ with $p > g+1$, the Roquette curve is the only curve that fails the Hurwitz bound.

We shall use the precise group structure of $\Aut(C)$, and also that all automorphisms are defined over $\F_{p^2}$ on $\RoqFpp := \RoqFp \otimes \F_{p^2}$, see Proposition~\ref{prop:allAutdefinedFpp}. As the existing literature does not seem to give the precise structure of this group, we refer to Appendix~\ref{app:A} for it. 

\begin{prop} \label{prop:irr}
For all  $\ell \not= p$, the representation 
\[
\rho_\ell \colon \Aut(C) \to \GL(\rH^1(C, \Q_\ell))
\]
is absolutely irreducible. 
\end{prop}

\begin{proof}
We denote by $N$ a $p$-Sylow subgroup of $\Aut(C)$. 
The dimension of $\rH^1(C, \Q_\ell)$ is $2g = (p-1)$, so that by Proposition~\ref{prop:irreducibleGrep} it is enough to check that $\rho_\ell|_N$ contains a non-trivial character, or equivalently that $\rho_\ell|_N$ is not trivial. This follows immediately from Proposition~\ref{prop:AUTfaithfulOnH1}. 

\end{proof}
It turns out that all we need from the Roquette curve is the geometric irreducibility proven in Proposition~\ref{prop:irr}.

\subsection{Curves with Schur obstruction}
\label{sec:curvesWithSchur}
Let $C$ be a smooth projective curve over $\bar \F_p$ of genus $g \ge 2$ such that the following holds:
\begin{enumerate}
\item[($\star$)] 
for all $\ell \not= p$ the representation of $G = \Aut(C)$ on $\rH^1(C,\Q_\ell)$ is 
absolutely  irreducible. 
\end{enumerate}
Let $J=J(C)$ be the Jacobian of $C$ and let $\rV_\ell(J) = \rT_\ell(J) \otimes_{\Z_\ell} \Q_\ell$ be the rational Tate module of $J$. 

\begin{lem}
\label{lem:SurjectiveEndJ}
Let $C$ be a smooth projective curve with $(\star)$.
For all  $\ell \not= p$, the natural map 
\[
\Q_\ell[G] \to \End(V_\ell(J))
\]
is surjective.
\end{lem}
\begin{proof}
Since  $\rH^1(C,\Q_\ell) = \Hom(V_\ell(J),\Q_\ell)$, 
the representation $G \to \GL(V_\ell(J))$ is dual to the representation on $\rH^1(C,\Q_\ell)$, which we assume to be absolutely irreducible. 
The claim follows from standard representation theory of finite groups.
\end{proof}

The following result is well known for the Roquette curve (\cite[p.~172]{Eke87} using slopes in crystalline cohomology) and in fact is a property shared by many curves with exceptionally large automorphism group.

\begin{prop}  \label{prop:ss}
Let $C$ be a smooth projective curve with $(\star)$. Then $C$ is supersingular. 
\end{prop}
\begin{proof}
Let $C_0/\F_q$ be a model of $C$ such that all automorphisms of $C$ are defined as automorphisms of $C_0$ over $\F_q$. Let $J_0$ be the Jacobian of $C_0$, so that $J =  J_0 \otimes_{\F_q} \bar F_p$. The geometric $q$-Frobenius of $C_0$ acts on $V_\ell(J_0) = V_\ell(J)$ commuting with $G$. The centralizer of the image of $G$ in $\End(V_\ell(J))$ consists only of scalars due to Lemma~\ref{lem:SurjectiveEndJ}. 

It follows that the $q$-Weil numbers associated to $J$ as the Jacobian of the curve $C$ defined over $\F_q$ are contained in a number field that admits an embedding to $\Q_\ell$ for all $\ell \not= p$. This must be $\Q$. The only $q$-Weil numbers that are rational are $\pm \sqrt{q}$, and $q$ must be a square. Since Frobenius acts as scalar, only one of the possible Weil numbers occurs as eigenvalue of Frobenius. By Honda-Tate theory, and because $q$ is a square, there is a supersingular elliptic curve $E_0$ over $\F_q$ with the same Weil number. We set $E  = E_0 \otimes_{\F_q} \bar \F_p$. It follows that 
\[
V_\ell(E^g) \simeq V_\ell(J)
\]
as Galois representations. By the Tate conjecture~\cite[Thm.~1]{Tat66} we find that $J$ and $E^g$ are isogenous, and that proves the claim.
\end{proof}

\begin{rmk}
Our main example will be the Roquette curve for which Proposition~\ref{prop:ss} has the following elementary shortcut. The hyperelliptic double cover $\RoqFp \to \P^1$ allows to count 
\[
\# \RoqFp(\F_p) = p+1.
\]
Concerning rational points over $\F_{p^2}$, we note that (1) they all lie over points in $\P^1(\F_{p^2}) \setminus \P^1(\F_p)$ and (2) the action of $G = \Aut(C)$ permutes all these possible images transitively. Since the hyperelliptic involution acts transtively on all fibres, we find that $\RoqFp(\F_{p^2})  \setminus \RoqFp(\F_p)$ is either empty or consists of $2 (p^2-p)$ many points. A precise calculation (which we omit because the precise description when which case occurs is irrelevant for us) shows 
\[
\# \RoqFp(\F_{p^2}) = \begin{cases}
p+1 & \textrm{ if } p \equiv 1 \pmod 4, \\
2p^2 - p + 1 & \textrm{ if } p \equiv 3 \pmod 4.
\end{cases}
\]
In any case, the Hasse-Weil bound for $\RoqFp$ and $\F_{p^2}$-rational points is sharp:
\[
| \#  \RoqFp(\F_{p^2}) - (1 + p^2)| = (p-1) \cdot p = 2g \sqrt{p^2}.
\]
In other words, the Roquette curve is minimal/maximal over $\F_{p^2}$, and this is only possible if the Frobenius eigenvalues are all $p$ or all $-p$. From here we argue as in the proof of Proposition~\ref{prop:ss}.
\end{rmk}

\begin{prop} \label{prop:Schur}
Let $C$ be a smooth projective curve with $(\star)$.  For $\ell \not= p$, the representation
\[
\rho_{J,\ell}: G \to \GL(\rV_\ell(J))
\]
has character with values in $\Z$ that is independent of $\ell$, but is not defined over $\Q$. 
The Schur index over $\Q$ is equal to $2$.
\end{prop}
\begin{proof}
Since $V_\ell(J)$ is dual to $\rH^1(C,\Q_\ell)$, the character has values in $\Z$ and is independent of $\ell$, for $\ell \not= p$, by the same argument as in the proof of Propositiion~\ref{prop:iindependenceofellgeometrriccase}.

\smallskip

Let $E$ be a supersingular elliptic curve as in the proof of
Proposition~\ref{prop:ss} such that $J$ is isogenous to $E^g$. We denote by 
\[
D = \End^0(E)
\]
the endomorphisms of $E$ over $\bar \F_{p}$ up to isogeny. This is the unique quaternion algebra over $\Q$ ramified in $p$ and $\infty$ only\footnote{Indeed, the action on the $2$-dimensional $\rH^1(E_{\bar \F_{p}}),\Q_\ell)$ shows that $D \otimes \Q_\ell \simeq \rM_2(\Q_\ell)$ for all $\ell \not= p$, and since $D$ is a skew field ($E$ is simple) and not commutative (that's in fact one possible definition of supersingular elliptic curve, see \cite[\S7]{Deu41}) there is no other central simple algebra over $\Q$ of dimension $4$ due to the local global principle for central simple algebras, see Brauer-Hasse-Noether \cite[Hauptsatz, Red.~1]{BHN32}.}, see \cite[\S8]{Deu41}.

The natural representation
\[
\Q[G] \to \End^0(J) \simeq \rM_{g}(\End^0(E)) = \rM_g(D)
\]
becomes, due to the Tate conjecture~\cite[Thm.~1]{Tat66}, under extension of scalars to $\Q_\ell$
\[
\Q_\ell[G] \to \End^0(J) \otimes \Q_\ell = \End_{\Gal}\big(V_\ell(J)\big) \subseteq \End\big(V_\ell(J)\big).
\]
Here $\Gal$ indicates Galois invariant endomorphisms. We know from Lemma~\ref{lem:SurjectiveEndJ} that the composition is surjective. So the inclusion on the right is in fact an equality (which also follows, because Frobenius was identified with a scalar in the  proof of Proposition~\ref{prop:ss}).
It follows that $\Q[G] \surj \rM_g(D)$ is surjective and identified with the component of the Wedderburn decomposition \eqref{eq:wedderburn} of the group ring corresponding to the irreducible representation underlying the $\rho_{J,\ell}$. Its Schur index is the Schur index of $D$, which indeed is $2$. 
\end{proof}

\section{A non-\texorpdfstring{$p'$}{prime to p}-discretely finitely generated fundamental group} \label{sec:ex}

The example  presented in this section rests on  Serre's construction \cite[\S15]{Ser58} (which he attributes to Weil~\cite[Chap.III]{Wei38}). 
Let $C$ be a smooth projective curve of genus $\ge 2$ over $\bar \F_p$ that satisfies $(\star)$ of 
Section~\ref{sec:curvesWithSchur}, and let 
 $G={\rm Aut}(C)$ be its group of automorphisms.  As a concrete example we can use the Roquette curve as discussed in Section~\ref{sec:Roquettecurve}. Let $P$ be a smooth  projective,  connected and  simply connected variety over $\bar \F_p$, such that $G$ acts freely on $P$, 
 see \cite[Prop.~15]{Ser58}. We define
 \ga{}{ X =(C \times_k P) /G,  \notag}
 where the action of $G$ on $C \times_k P$ is the diagonal action.

\begin{thm} \label{thm:ex}
The fundamental group $\pi_1(X)$ is not $p'$-discretely finitely presented. 
\end{thm}

Applying Proposition~\ref{prop:crit} we obtain the following.

 \begin{cor} \label{cor:nonlift}
 The variety $X$  does not lift to characteristic $0$. 
 \end{cor}

  In particular, the condition for $\pi_1(X)$ to be $p'$-finitely discretely presented is a (new) obstruction for a characteristic $p$ smooth proper geometrically irreducible variety defined over an algebraically closed characteristic $p>0$ field to be liftable to characteristic $0$.
  
  \begin{proof}[Proof of Theorem~\ref{thm:ex}.]
As $G$ acts freely on $P$, the finite morphism $C \times_k P \to X$ is Galois \'etale of group $G$.  With 
$ \pi_1(C\times_k  P)=\pi_1(C)$,  due to the K\"unneth formula,
  Galois theory induces an exact sequence
  \ga{}{1\to \pi_1(C) \to \pi_1(X) \xrightarrow{\ph} G\to 1.\notag}
The outer $G$ action on $U_\ph = \pi_1(C)$ by conjugation considered in \eqref{eq:outerGaction} agrees with the natural action by the functor $\pi_1$ (without base point). The associated $\ell$-adic representations 
\[
\rho_\ell : G \to \GL(U_\ph^\ab \otimes \Q_\ell)
\] 
as considered in \eqref{eq:theelladicrep} therefore agree with the natural representations on $V_\ell(J)$, the rational Tate module of the Jacobian $J$ of $C$. 

It now follows from Proposition~\ref{prop:Schur} that $\rho_\ell$ is independent of $\ell$ but of Schur index $2$. Therefore $\pi_1(X)$ fails to be Schur rational and Proposition~\ref{prop:discretelyfgimpliesSchur} shows that $\pi_1(X)$ is not $p'$-discretely finitely generated.
\end{proof}

\appendix 
\section{The automorphism group of the Roquette curves} \label{app:A}

Recall from  Section~\ref{sec:manyaut} that the Roquette curve $\RoqFp$ over $\F_p$ is the smooth hyperelliptic curve of genus obtained as the  compactification of the affine curve defined by the equation
\[
y^2 = x^p - x.
\]
The Roquette curve $\RoqFp$ has genus $g = (p-1)/2$, so $g \ge 2$ if and only if $p\ge 5$. We are going to construct a finite group $G$, define an action of $G$ on $\RoqFpp = \RoqFp \otimes \F_{p^2}$, and show that $G$ is the full group of automorphisms of $\Roq = \RoqFp \otimes \bar \F_p$.

\smallskip

From now on we assume $p \geq 5$.

\subsection{The automorphisms}

The group of square roots 
\[
(\F_p^\times)^{1/2} := \{ \lambda \in \bar \F_{p}^\times \ ; \ \lambda^2 \in \F_p^\times\}
\]
is a cyclic subgroup of $\F_{p^2}^\times$ of order $2(p-1)$.  We define the group $\tilde G$ as the fibre product
\[
\tilde{G} := \{ (A,\lambda) \in  \GL_2(\F_p) \times (\F_{p}^\times)^{1/2} \ ; \ \det(A) = \lambda^2\}.
\]
The action of $\tilde{G}$ on $\RoqFpp$ arises as follows.  Let 
$g=(A,\lambda) \in  \tilde{G} \subseteq \GL_2(\F_p) \times \F_{p^2}^\times$ with matrix part 
$A=\begin{pmatrix}a & b \\  c&  d  \\ \end{pmatrix}$. 
Then we denote by $\alpha_g$ the map $\RoqFpp \to \RoqFpp$ defined in coordinates by
\begin{align*}
\alpha_g^\ast(x) &  := A(x) := \frac{ax + b}{cx+d}, \\ 
\alpha_g^\ast(y) & :=  \frac{\lambda \cdot y}{(cx+d)^{(p+1)/2}}. 
\end{align*}
Here $A(x)$ is the usual M\"obius action.

\begin{prop} \label{prop:GactionC}
The map $g \mapsto \alpha_g$ defined above yields a group homorphism
\[
\alpha: \tilde{G} \to \Aut_{\F_{p^2}}(\RoqFpp) .
\]
\end{prop}
\begin{proof}
For $g = (A,\lambda) \in \tilde{G} $, with $A=\begin{pmatrix}a & b \\  c&  d  \\ \end{pmatrix}$, indeed $\alpha_g$ defines a map $\RoqFpp \to \RoqFpp$;
\begin{align*}
\alpha_g^\ast(y)^2  & = \frac{ \lambda^2 \cdot y^2}{(cx+d)^{p+1}} =  \frac{\det(A) \cdot(x^p - x)}{(cx+d)^{p+1}}  = \frac{(a x^p + b)(cx+d) - (ax+b)(cx^p+d)}{(cx^p+d)(cx+d)}  \\
& = \frac{ax^p + b}{cx^p + d} -  \frac{ax + b}{cx + d} = \alpha_g^\ast(x)^p - \alpha_g^\ast(x).
\end{align*}

For another element $h = (B,\mu) \in \tilde{G}$ with matrix part $B = \begin{pmatrix} a' & b' \\  c' &  d' \end{pmatrix}$ we compute
\[
\alpha_h^\ast(\alpha_g^\ast(x)) = B(A(x)) = (AB)(x) = \alpha_{gh}^\ast(x),
\]
because $\GL_2$ acts by M\"obius transformations on $\P^1$. Moreover, 
\begin{align*}
\alpha_h^\ast(\alpha_g^\ast(y) & = \alpha_h^\ast\big(\frac{\lambda \cdot  y}{(c x+d)^{(p+1)/2}}\big) 
= \frac{\lambda  \cdot  \frac{\mu \cdot y}{(c'  x+d')^{(p+1)/2}}}{(c \alpha_h^\ast(x) + d)^{(p+1)/2}}  = \frac{\lambda \mu  \cdot  y}{\big( (c \alpha_h^\ast(x) + d)(c'  x+d') \big)^{(p+1)/2}}   \\
& =\frac{\lambda \mu  \cdot  y}{\big( (c (a'x + b') + d(c' x+ d') \big)^{(p+1)/2}}   
 =\frac{\lambda \mu  \cdot  y}{\big( (ca' + d c') x + (cb' + dd') \big)^{(p+1)/2}}  = \alpha_{gh}^\ast(y).
\end{align*}
Since $\alpha_{(\I,1)}$ is the identity on $\RoqFpp$, 
here $\I$ is the unit matrix, the above shows simultaneously that $\alpha_g$ is an automorphism and $\alpha$ is a homomorphism.
\end{proof}

Let $\iota : \RoqFp \to \RoqFp$ be the hyperelliptic involution $(x,y) \mapsto (x,-y)$. Since $\iota$ acts as $-1$ on the Jacobian of $\RoqFp$, it centralizes all automorphisms of $C_k = \RoqFp \otimes k$ for any field $k$. In particular, any  automorphism $f: C_k \to C_k$ descends to a map $\bar f : \P^1_k \to \P^1_k$.
Since the ramification locus of the hyperelliptic covering $x : C_k \to \P_k^1$ consists of all $\F_p$-rational points, the induced map $\bar f$ must permute these. Therefore the M\"obius transformation describing $\bar f$ has matrix entries in $\F_p$ due to the following lemma. 

\begin{lem} \label{lem:PGL2p}
Let $k$ be a field of characteristic $p$. The group of automorphisms of $\P^1_k$ that permutes the subset $\P^1(\F_p)$ consists of the M\"obius transformations from $\PGL_2(\F_p)$.
\end{lem}
\begin{proof}
The group $\PGL_2(k)$ acts sharply $3$-transitively on $\P^1(k)$ for all fields $k$.
\end{proof}

Let $k$ be a field containing $\F_{p^2}$. Then we deduce from Proposition~\ref{prop:GactionC} and Lemma~\ref{lem:PGL2p} a commutative diagram:
\begin{equation}
\label{eq:GcompareAut}
\xymatrix@M+1ex{
1 \ar[r] & \langle (\I, -1)\rangle \ar[r] \ar[d] & \tilde{G} \ar[d]^\alpha \ar[r]^{\pr_1} & \GL_2(\F_p) \ar[d] \ar[r] & 1\ \\
1 \ar[r] & \langle \iota \rangle \ar[r] & \Aut_k(C_k) \ar[r]^{f \mapsto \bar f} & \PGL_2(\F_p)   \ar[r] & 1.
}
\end{equation}
Here $\pr_1$ is the projection $(A,\lambda) \mapsto A$. 

\begin{prop}
\label{prop:allAutdefinedFpp}
Both rows of \eqref{eq:GcompareAut} are exact and the vertical maps are surjective. In particular, all the automorphisms of a Roquette curve are defined over $\F_{p^2}$.
\end{prop}
\begin{proof}
The top row is exact because squaring is surjective as a map $(\F_p^\times)^{1/2} \surj \F_p^\times$.  The bottom row is left exact by Galois theory of the hyperelliptic cover $C_k \to \P^1_k$, and we are going to show that the map $f \mapsto \bar f$ is also surjective. 
The left vertical map is an isomorphism because of $\alpha((\I,-1)) = \iota$. The right vertical map is the natural projection and thus also surjective. It follows that the bottom row is also exact and that $\alpha$ is surjective. 

\end{proof}

Let $\qr{\lambda}{p} = \lambda^{(p-1)/2} \in \{\pm 1\}$ denote the Legendre quadratic residue symbol modulo $p$. Then 
\[
\lambda^{(p+1)/2} = \qr{\lambda}{p}\lambda
\]
and we have an injective group homomorphism
\[
\F_p^\times \to \tilde{G}, \qquad \lambda \mapsto (\lambda \I, \qr{\lambda}{p}\lambda),
\]
because $\det(\lambda \I) = \lambda^2 = (\qr{\lambda}{p}\lambda)^2$. All $(\lambda \I, \qr{\lambda}{p}\lambda)$ are contained in the kernel of $\alpha$. So a diagram chase with \eqref{eq:GcompareAut} shows the following.

\begin{prop}
\label{prop:structureG}
Let $k$ be a field containinig $\F_{p^2}$. 
The homomorphism $\alpha$ induces an isomorphism 
\[
G := \tilde{G}/\{ (\lambda \I, \qr{\lambda}{p}\lambda) \ ; \ \lambda \in \F_p^\times \} \xrightarrow{\sim} \Aut_k(C_k).
\]
\end{prop}

It follows that the Roquette curve $C$ has $2p(p^2-1)$ many automorphisms, see \cite[Section~4]{Roq70}. The main result of \textit{loc.\ cit.}\ shows that among all curves with $p > g+1$ the Roquette curve is the only curve violating the Hurwitz bound $84(g-1)$ for the order of the automorphism group.

\subsection{Basic representation theory of \texorpdfstring{$G$}{G}} 
We denote by $N$ the image in $G$ of the group of upper triangular unipotent matrices
\[
N = \im\big( \{\begin{pmatrix} 1 & u \\ 0 & 1\end{pmatrix} \ ; \ u \in \F_p\} \inj \SL_2(\F_p) \xhookrightarrow{A \mapsto (A,1)} \tilde{G} \surj G\big).
\]
The group $N$ is cyclic of order $p$ and thus a $p$-Sylow of $G$. 

\begin{lem} 
\label{lem:pelementsallconjugateinG}
All elements of order $p$ in $G$ are conjugate to each other.
\end{lem}
\begin{proof} 
The computation in $\GL_2(\F_p)$ 
\[
\matzz{m}{0}{0}{1} \matzz{1}{1}{0}{1} \matzz{m^{-1}}{0}{0}{1} = \matzz{1}{m}{0}{1} 
\]
shows that in $\GL_2(\F_p)$ all elements of order $p$ are conjugate to each other. Indeed, any element of order $p$ is conjugate to an element of the upper triangular unipotent matrices by  Sylow's theorems, and the computation explains the rest. 

The same holds in $G$ although $\GL_2(\F_p)$ is not a subgroup of $G$. Again by  Sylow's theorems  we only have to prove the lemma for nontrivial $(s,1),(t,1) \in N$. Then, from the $\GL_2$-result, we know that there is a matrix $A \in \GL_2(\F_p)$ with $AsA^{-1} = t$. Now we choose a square root $\lambda$ of $\det(A)$. The element $g \in G$ which is the image of $(A,\lambda) \in \tilde{G}$ does the job:
\[
g(s,1)g^{-1} = (A,\lambda)(s, 1)(A^{-1},\lambda^{-1}) = (AsA^{-1},1) = (t,1). \qedhere
\]
\end{proof}

Let  $K$ be a field of characteristic $0$. We consider a representation $\rho : G \to \GL(V)$ with a finite dimensonal $K$ vector space $V$. For simplicity we assume that $K$ contains the $p$-th roots of unity. The restriction $V|_N$ to $N$ decomposes into a direct sum of $1$-dimensional representations, according to the $K$-valued characters  
$\psi : N \to K^\times$ on $N$.

\begin{prop} \label{prop:samemultiplicity}
In the situation above, the multiplicity of $\psi$ occuring in $(V,\rho)$ is the same for all non-trivial characters $\psi$.
\end{prop}
\begin{proof} 
Let $\chi$ be the character of $\rho$ as a representation of $G$. By Lemma~\ref{lem:pelementsallconjugateinG} the value of $\chi$ on $N \setminus\{1\}$ is constant, say $\chi(s) = n_\chi$. The multiplicity of $\psi$ in $V|_N$ is computed as
\begin{align*}
\langle \res_{N}(\chi), \psi \rangle_N & = \frac{1}{|N|} \cdot \sum_{s \in N} \chi(s) \psi(s^{-1})  = \frac{1}{|N|} (\chi(1) - n_\chi) + n_\chi  \frac{1}{|N|} \cdot \sum_{s \in N}  \psi(s^{-1})  \\
& = \frac{1}{|N|} (\chi(1) - n_\chi) + n_\chi \langle \one, \psi \rangle_N = \frac{1}{|N|} (\chi(1) - n_\chi).
\end{align*}
Here $\one$ is the trivial representation, and the vanishing of $\langle \one, \psi \rangle_N$ follows from the orthogonality relations since $\psi$ is non-trivial, or even from more elementary facts on characters.
\end{proof}

\begin{prop}
\label{prop:irreducibleGrep}
Let $(V,\rho)$ be a representation of $G$ such that the restriction $V|_N$ is not the trivial representation. 
Then  $\dim_K(V) \geq (p-1)$, and if equality occurs, then $\rho$ is an absolutely irreducible representation.
\end{prop}
\begin{proof}
The assumption $V|_N$ non-trivial means that there is a non-trivial character $\psi$ of $N$ that occurs on $V|_N$. There are $(p-1)$ non-trivial characters of $N$, and each occurs in $V|_N$ with the same multiplicity according to Proposition~\ref{prop:samemultiplicity}. The dimension estimate follows at once. 

We can apply the same reasoning to an irreducible subrepresentation $W \subseteq V$, and we may choose one  which contains a nontrivial character $\psi$ of $N$. The dimension estimate in case of $\dim_K(V) = (p-1)$ shows $V=W$, hence $V$ itself is irreducible. The same argument applies after scalar extension to  an algebraic closed field, hence the representation is even absolutely irreducible.
\end{proof}

\begin{rmk}
Proposition~\ref{prop:irreducibleGrep} applies in particular to a faithful $G$-representation.
\end{rmk}

 
\end{document}